\documentclass[letterpaper]{icm2010-v1-mod}

\usepackage{graphicx}

\title[Order and Disorder in Energy Minimization]{Order and Disorder in Energy Minimization}
\author[Henry Cohn]{Henry Cohn\thanks{Microsoft Research New England, One Memorial Drive,
Cambridge, MA 02142, USA. E-mail: cohn@microsoft.com.}}

\newcommand{\R}{\mathbb{R}}
\newcommand{\C}{\mathbb{C}}
\newcommand{\Z}{\mathbb{Z}}
\newcommand{\F}{\mathbb{F}}
\newcommand{\Ps}{\mathbb{P}}
\newcommand{\Pol}{\mathcal{P}}

\newcommand{\vol}{\mathop{\textup{vol}}}

\newtheorem{theorem}{Theorem}[section]

\newtheorem{lemma}[theorem]{Lemma}

\newtheorem{conjecture}[theorem]{Conjecture}
\theoremstyle{definition}

\numberwithin{equation}{section}

\begin{document}

\setcounter{page}{2416}

\begin{abstract}
How can we understand the origins of highly symmetrical objects? One
way is to characterize them as the solutions of natural optimization
problems from discrete geometry or physics.  In this paper, we explore
how to prove that exceptional objects, such as regular polytopes or the
$E_8$ root system, are optimal solutions to packing and potential
energy minimization problems.
\end{abstract}

\begin{Classification}
Primary 05B40, 52C17; Secondary 11H31.
\end{Classification}


\maketitle


\section{Introduction}

\subsection{Genetics of the regular figures}

Symmetry is all around us, both in the physical world and in
mathematics.  Of course, only a few of the many possible symmetries are
ever actually realized, but we see more of them than we seemingly have
any right to expect: symmetry is by its very nature delicate, and
easily disturbed by perturbations.  It is no great surprise to see
carefully designed, symmetrical artifacts, but it is remarkable that
nature can ever produce similar effects robustly, for example in
snowflakes. Any occurrence of symmetry not deliberately imposed demands
an explanation.

L\'aszl\'o Fejes T\'oth proposed to seek the origins of symmetry in
optimization problems.  He referred to the \emph{genetics of the
regular figures}, in which ``regular arrangements are generated from
unarranged, chaotic sets by the ordering effect of an economy
principle, in the widest sense of the word'' \cite{FT}.  It is not
enough simply to classify the possible symmetries; we must go further
and identify the circumstances in which they arise naturally.

Over the last century mathematicians have made enormous progress in
identifying possible symmetry groups.  We have classified the simple
Lie algebras and finite simple groups, and although there is much left
to learn about group theory and representation theory, our collective
knowledge is both extensive and broadly applicable. Unfortunately, our
understanding of the genetics of the regular figures lags behind.  Much
is known, but far more remains to be discovered, and many natural
questions seem totally intractable.

Optimization provides a framework for this problem.  How much symmetry
and order should we expect in the solution of an optimization problem?
It is natural to guess that the solutions of a highly symmetric problem
will inherit the symmetry of the problem, but that is not always the
case.  For a toy example, consider the Steiner tree problem for a
square, i.e., how to connect all four vertices of a square to each
other via curves with minimal total length.  The most obvious guess
connects the vertices by an X, which displays all the symmetries of the
square, but it is suboptimal. Instead, in the optimal solutions the
branches meet in threes at $120^\circ$ angles (this is a
two-dimensional analogue of the behavior of soap films):
\begin{center}
\includegraphics[scale=0.65]{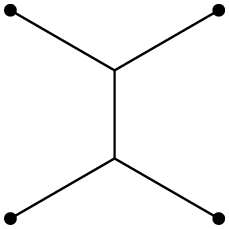} \qquad \qquad \qquad
\includegraphics[scale=0.65,angle=90]{steiner.eps}
\end{center}
Note that the symmetry of the square is broken in each individual
solution, but of course the set of both solutions retains the full
symmetry group.

It is tempting to use symmetry to help solve problems, or at least to
guess the answers, but as the Steiner tree example shows, this approach
can be misleading.  One of the most famous mistaken cases was the
Kelvin conjecture on how to divide three-dimensional space into
infinitely many equal volumes with minimal surface area between them,
to create a foam of soap bubbles.  In 1887 Kelvin conjectured a simple,
symmetrical solution, obtained by deforming a tiling of space with
truncated octahedra.  (The deformation slightly curves the hexagonal
facets into monkey saddles, so that the foam has the appropriate
dihedral angles.)  Kelvin's conjecture stood unchallenged for more than
a century, but in 1994 Weaire and Phelan found a superior solution with
two irregular types of bubbles\footnote{Their foam structure was the
inspiration for the Beijing National Aquatics Center, used in the
$2008$ Olympics.} \cite{WP}. This shows the danger of relying too much
on symmetry: sometimes it is a crucial clue as to the true optimum, but
sometimes it leads in the wrong direction.

In many cases the symmetries that are broken are as interesting as the
symmetries that are preserved.  For example, crystals preserve some of
the translational symmetries of space, but they dramatically break
rotational symmetry, as well as most translational symmetries.  This
symmetry breaking is remarkable, because it entails long-range
coordination: somehow widely separated pieces of the crystal
nevertheless align perfectly with each other.  A complete theory of
crystal formation must therefore deal with how this coordination could
come about.  Here, however, we will focus on optimization problems and
their solutions, rather than on the physical or algorithmic processes
that might lead to these solutions.

\vspace*{-6pt}

\subsection{Exceptional symmetry: $E_8$ and the Leech lattice}

Certain mathematical objects, such as the icosahedron, have always
fascinated mathematicians with their elegance and symmetry. These
objects stand out as extraordinary and have inspired much deep
mathematics (see, for example, Felix Klein's \emph{Lectures on the
Icosahedron\/} \cite{K}).  They are the sorts of objects one hopes to
characterize and understand via the genetics of the regular\break figures.

These objects are often exceptional cases in classification theorems.
In many different branches of mathematics, highly structured or
symmetric objects can be classified into several regular, predictable
families together with a handful of exceptions, such as the exceptional
Lie algebras or sporadic finite simple groups.  For most applications,
the infinite families play the leading role, and one might be tempted
to dismiss the exceptional cases as aberrations of limited importance,
specific to individual problems. Instead, although they are indeed
peculiar, the exceptional cases are not merely isolated examples, but
rather recurring themes throughout mathematics, with the same
exceptions occurring in seemingly unrelated problems. This phenomenon
has not yet been fully understood, although much is known about
particular cases.

For example, $ADE$ classifications (i.e., simply-laced Dynkin diagrams)
occur in many different mathematical areas, including finite subgroups
of the rotation group $SO(3)$, representations of quivers of finite
type, certain singularities of algebraic hypersurfaces, and simple
critical points of multivariate functions. In each case, there are two
infinite families, denoted $A_n$ and $D_n$, and three exceptions $E_6$,
$E_7$, and $E_8$, with each type naturally described by a certain
Dynkin diagram. See \cite{HHSV} for a survey.  This means $E_8$, for
example, has a definite meaning in each of these problems.  For
example, among rotation groups it corresponds to the icosahedral group,
and among simple critical points of functions from $\R^n$ to $\R$ it
corresponds to the behavior of $ x_1^3 + x_2^5 + x_3^2 + x_4^2 + \dots
+ x_{n}^2 $ at the origin.

In this survey, we focus primarily on two exceptional structures,
namely the $E_8$ root lattice in $\R^8$ and the Leech lattice in
$\R^{24}$.  These objects bring together numerous mathematical{\pagebreak} topics,
including sphere packings, finite simple groups, combinatorial and
spherical designs, error-correcting codes, lattices and quadratic
forms, mathematical physics, harmonic analysis, and even hyperbolic and
Lorentzian geometry.  They are far too rich and well connected to do
justice to here; see \cite{CS} for a much longer account as well as
numerous references. Here, we will examine how to characterize $E_8$
and the Leech lattice, as well as some of their relatives, by
optimization problems.  These objects are special because they solve
not just a single problem, but rather a broad range of problems. This
level of breadth and robustness helps explain the widespread
occurrences of these structures within mathematics.  At the same time,
it highlights the importance of understanding which problems have
extraordinarily symmetric solutions and which do not.

\subsection{Energy minimization}

Much of physics is based on the idea of energy minimization, which will
play a crucial role in this article. In many systems energy dissipates
through forces such as friction, or more generally through heat
exchange with the environment. Exact energy minimization will occur
only at zero temperature; at positive temperature, a system in contact
with a heat bath (a vast reservoir at a constant temperature, and with
effectively infinite heat capacity) will equilibrate to the temperature
of the heat bath, and its energy will fluctuate randomly, with its
expected value increasing as the temperature increases.

One can describe the behavior of such a system mathematically using
\emph{Gibbs measures}, which are certain probability distributions on
its states. For simplicity, imagine a system with $n$ different states
numbered $1$ through $n$, where state $i$ has energy $E_i$. For each
possible expected value $\bar{E}$ for energy, the corresponding Gibbs
measure is the maximal entropy probability measure constrained to have
expected energy $\bar{E}$.  In other words, it assigns probability
$p_i$ to state $i$ so that the entropy $ \sum_{i=1}^n -p_i \log p_i $
is maximized subject to $ \sum_{i=1}^n p_i E_i = \bar{E}. $  (For the
motivation behind the definition of entropy, see \cite{Kh}.)

A Lagrange multiplier argument shows that when $\min_i E_i < \bar{E} <
\max_i E_i$, the probability $p_i$ must equal $e^{-\beta
E_i}/\sum_{j=1}^n e^{-\beta E_j}$ for some constant $\beta$, where
$\beta$ is chosen so that the expected energy equals $\bar{E}$.  In
physics terms, $\beta$ is proportional to the reciprocal of
temperature, and only nonnegative values of $\beta$ are relevant
(because energy is usually not bounded above, as it is in this toy
model).  As the temperature tends to infinity, $\beta$ tends to zero
and the system will be equidistributed among all states.  As the
temperature tends to zero, $\beta$ tends to infinity, and the system
will remain in its \emph{ground states}, i.e., those with the lowest
possible energy.

In this article, we will focus on systems of point particles
interacting via a pair potential function.  In other words, the energy
of the system is the sum over all pairs of particles of some function
depending only on the relative position of the pair (typically the
distance between them).  For example, in classical electrostatics, it
is common to study identical charged particles interacting via the
Coulomb potential, i.e., with potential energy $1/r$ for a pair of
particles at distance $r$.

Many other mathematical problems can be recast in this form, even
sometimes in ways that are not immediately apparent.  For a beautiful
although tangential example, consider the distribution of eigenvalues
for a random $n \times n$ unitary matrix, chosen with respect to the
Haar measure on $U(n)$. These eigenvalues are unit complex numbers
$z_1,\dots,z_n$, and the Weyl integral formula says that the induced
probability measure on them has density proportional to
$$
\prod_{1 \le i < j \le n} |z_i - z_j|^2
$$
(see \cite{Dy}).  If we define the logarithmic potential $-\log
|z_i-z_j|$ between $z_i$ and $z_j$, then this measure is the Gibbs
measure with $\beta = 2$ for $n$ particles on the unit circle.  The
logarithmic potential is natural because it is a harmonic function on
the plane (much as the Coulomb potential $x \mapsto 1/|x|$ is harmonic
in three dimensions). Thus, the eigenvalues of a random unitary matrix
repel each other through harmonic interactions, and the Weyl integral
formula specifies the temperature $1/\beta$.

In the following survey, we will focus on the case of zero temperature.
In the real world, all systems have positive temperature, which raises
important questions about dynamics and phase transitions. However, for
the purposes of understanding the role of symmetry, zero temperature is
a crucial case.

\subsection{Packing and information theory}
\label{subsec:packinginformation}

The prototypical packing problem is sphere packing: how can one arrange
non-overlapping, congruent balls in Euclidean space to fill as large a
fraction of space as possible?  The fraction of space filled is the
\emph{density}. Of course, it must be defined by a limiting process, by
looking at the fraction of a large ball or cube that can be covered.

Packing problems fit naturally into the energy minimization framework
via \emph{hard-core potentials}, which are potentials that are infinite
up to a certain radius $r$ and zero at or beyond it. In other words,
there is an infinite energy penalty for points that are too close
together, but otherwise there is no effect. Under such a potential
function, a collection of particles has finite energy if and only if
the particles are positioned at the centers of non-overlapping balls of
radius $r/2$. Note that every packing (not just the densest) minimizes
energy, but knowing the minimal energy for all densities solves the
packing problem.

From this perspective, one can formulate questions that are even deeper
than densest packing questions.  For example, at any fixed density, one
can ask for a random packing at that density (i.e., a sample from the
Gibbs measure at zero temperature).  For which densities is there
long-range order, i.e., nontrivial correlations between distant
particles? In two or three dimensions, the densest packings are
crystalline, and there appears to be considerable order even below the
maximal density, with a phase transition between order and disorder as
the density decreases. (See \cite{L} and the references cited therein
for more details.)  It is far from clear what happens in high
dimensions, and the densest packings might be disordered \cite{TS}.

Packings of less than maximal density are of great importance for
modeling granular materials, because most such materials will be
somewhat loose. The fact that long-range order seemingly persists over
a range of densities means it can potentially be observed in the real
world, where even under high pressure no packing is ever truly perfect.
(Of course, for realistic models there are many other important
refinements, such as variation in particle sizes and\break shapes.)

In addition to being models for granular materials, packings play an
important role in information theory, as error-correcting codes for
noisy communication channels.  Suppose, for a simplified example, that
we wish to communicate by radio.  We can measure the signal strength at
$n$ different frequencies and represent it as an $n$-dimensional
vector.  Note that $n$ may be quite large, so high-dimensional packings
are especially important here.  The power required to transmit a signal
$x \in \R^n$ will be proportional to $|x|^2$, so we must restrict our
attention to signals that lie within a ball of radius $r$ centered at
the origin, where $r$ depends on the power level of our transmitter.

If we transmit a signal, then the received signal will be slightly
perturbed due to noise.  We can measure the noise level of the channel
by $\varepsilon$, so that when $x$ is transmitted, with high
probability the received signal $x'$ will satisfy $|x-x'| <
\varepsilon$.  In other words, if the open balls of radius
$\varepsilon$ about signals $x$ and $y$ do not overlap, then with high
probability the received signals $x'$ and $y'$ cannot be confused.

To ensure error-free communication, we will rely on a restricted
vocabulary of possible signals that cannot be confused with each other
(i.e., an error-correcting code). That means they must be the centers
of non-overlapping balls of radius $\varepsilon$. For efficient
communication, we wish to maximize the number of signals available for
use, i.e., the number of such balls whose centers lie within a ball of
radius $r$.  In the limit as $r/\varepsilon$ tends to infinity, that is
the sphere packing problem.

\subsection{Outline}
The remainder of the paper is organized as follows.
Sections~\ref{section:packings} and~\ref{section:thomson} survey
packing and energy minimization problems in more depth.
Sections~\ref{section:compact} and~\ref{section:euclidean} outline the
proofs that certain exceptional objects solve these problems. Finally,
Section~\ref{section:future} offers areas for future investigation.

\section{Packings and Codes} \label{section:packings}

\subsection{Sphere packing in low and high dimensions}

\begin{figure}
\vspace*{-3pt}
\begin{center}
\includegraphics[scale=0.75]{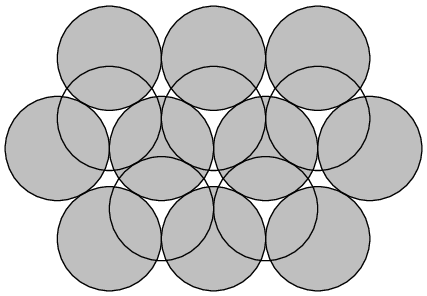}
\vspace*{-17pt}
\end{center}
\caption{Two layers in a three-dimensional sphere packing, one denoted
by shaded circles and the other by unshaded circles.  Notice that the
unshaded layer sits above half of the holes in the shaded layer.}
\label{figure:hexagonal}
\vspace*{-3pt}
\end{figure}

One can study the sphere packing problem in any dimension.  In $\R^1$
it is trivial, because the line can be completely covered with
intervals. In $\R^2$, it is easy to guess that a hexagonal arrangement
of circles is optimal, with each circle tangent to six others, but
giving a rigorous proof of optimality is not completely straightforward
and was first achieved in 1892 by Thue \cite{Th} (see \cite{H1} for a
short, modern proof).  In $\R^3$, the usual way oranges are stacked in
grocery stores is optimal, but the proof is extraordinarily difficult.
Hales completed a proof in 1998, with a lengthy combination of human
reasoning and computer calculations \cite{H}.  One conceptual
difficulty is that the solution is not at all unique in $\R^3$.  In a
technical sense, it is not unique in any dimension (even up to
isometries), because density is a global property that is unchanged by,
for example, removing a ball. However, in three dimensions there is a
much deeper sort of non-uniqueness.  One can form an optimal packing by
stacking hexagonal layers, with each layer nestled into the gaps in the
layer beneath it. As shown in Figure~\ref{figure:hexagonal}, the holes
in a hexagonal lattice consist of two translates of the original
lattice, and the next layer will sit above one of these two translates.
For each layer, a binary choice must be made, and there are uncountably
many ways to make these choices. (Each will be isometric to countably
many others, but there remain uncountably many geometrically distinct
packings, with many different symmetry groups.) All these packings are
equally dense and perfectly natural.  See \cite{CS1} for a discussion
of this issue in higher\break dimensions.

In four or more dimensions, no sharp density bounds are known. Instead,
we merely have upper and lower bounds, which differ by a substantial
factor.  For example, in $\R^{36}$, the best upper bound known is more
than $58$ times the density of the best packing known \cite{CE}.  This
factor grows exponentially with the dimension: the best lower bound
known is a constant times $n2^{-n}$ in $\R^n$ (see \cite{B} and
\cite{V}), while the upper bound is $(1.514724\ldots + o(1))^{-n}$ (see
\cite{KL}).

It may be surprising that these densities are so low.  One way to think
about it is in terms of volume growth in high dimensions.  An
$\varepsilon$-neighborhood of a ball in $\R^n$ has volume
$(1+\varepsilon)^n$ times that of the ball, so when $n$ is large, there
is far more volume near the surface of the ball than actually inside
it. In low-dimensional sphere packings, most volume is contained within
the balls, with a narrow fringe of gaps between them. In
high-dimensional packings, the gaps occupy far more volume.

It is easy to prove a lower bound of $2^{-n}$ for the sphere packing
density in $\R^n$.  In fact, this lower bound holds for every
\emph{saturated packing} (i.e., one in which there is no room for any
additional spheres):

\begin{lemma} \label{lemma:saturated}
Every saturated sphere packing in $\R^n$ has density at least $2^{-n}$.
\end{lemma}

\begin{proof}
Suppose the packing uses spheres of radius $r$. No point in space can
be distance $2r$ or further from the nearest sphere center, since
otherwise there would be room to center another sphere of radius $r$ at
that point. This means we can cover space completely by doubling the
radius of each sphere. Doubling the radius multiplies the volume by
$2^n$, and hence multiplies the density by at most $2^n$ (in fact,
exactly $2^n$ if we count overlaps with multiplicity). Because the
enlarged spheres cover all of space, the original spheres must cover at
least a $2^{-n}$ fraction.
\end{proof}

This argument sounds highly constructive (simply add more spheres to a
packing until it becomes saturated), and indeed it is constructive in
the logical sense.  However, in practice it offers almost no insight
into what dense packings look like, because it is difficult even to
tell whether a high-dimensional packing is saturated.

In fact, it is completely unclear how to construct dense packings in
high dimensions.  One might expect the sphere packing problem to have a
simple, uniform solution that would work in all dimensions.  Instead,
each dimension has its own charming idiosyncrasies, as we will see in
Section~\ref{subsec:lattices}. There is little hope of a systematic
solution to the sphere packing problem in all dimensions. Even
achieving density $2^{-n}$ through a simple, explicit construction is
an unsolved problem.

\subsection{Lattices and periodic packings} \label{subsec:lattices}

The simplest sorts of packings are lattice packings.  Recall that a
\emph{lattice} in $\R^n$ is the integral span of a basis (i.e., it is a
grid, possibly skewed).  To form a sphere packing, one can center a
sphere at each lattice point.  The radius should be half the minimal
distance between lattice points, so that the nearest spheres are
tangent to each other.

There is no reason to expect that lattice packings should be the
densest sphere packings, and they are probably not optimal in
sufficiently high dimensions (for example, ten dimensions).  However,
lattices are very likely optimal in $\R^n$ for $n \le 9$ and for some
higher values of $n$ (including $12$, $16$, and $24$). See \cite{CS}
for more details about lattices and packings in general.

For $n \le 8$ and $n=24$, the lattice packing problem has been solved
in $\R^n$. In fact, the densest lattices are unique in these dimensions
(up to scaling and isometries), although that may not be true in every
dimension, such as $n=25$.  For $n \le 8$, the optimal lattices are all
root lattices, the famous lattices that arise in Lie theory and are
classified by Dynkin diagrams. Specifically, the densest lattices are
$A_1$ (the integer lattice), $A_2$ (the hexagonal lattice), $A_3$ (the
face-centered cubic lattice, which is also isomorphic to $D_3$), $D_4$,
$D_5$, $E_6$, $E_7$, and $E_8$.  For $n=24$, the Leech lattice is an
optimal lattice packing; the proof will be discussed in
Section~\ref{section:euclidean}.

The $D_n$ lattices are particularly simple, because they are formed by
a checkerboard construction as a sublattice of index $2$ in $\Z^n$:
$$
D_n = \{(x_1,\dots,x_n) \in \Z^n : x_1+\dots+x_n \equiv 0 \pmod{2}\}.
$$
To see why $D_n$ is not optimal in high dimensions, consider the
\emph{holes} in $D_n$, i.e., the points in space that are local maxima
for distance from the lattice. The integral points with odd coordinate
sum are obvious candidates, and they are indeed holes, at distance $1$
from $D_n$. However, there's a slightly more subtle case, namely the
point $(1/2,1/2,\dots,1/2)$ and its translates by $D_n$. These points
are at distance
$$
\sqrt{\left(\frac{1}{2}\right)^2 + \dots + \left(\frac{1}{2}\right)^2}
= \sqrt{n/4}
$$
from $D_n$.  When $n=8$, this distance becomes $\sqrt{2}$, which is
equal to the minimal distance between points in $D_8$.  That means
these deep holes have become large enough that additional spheres can
be placed in them.  Doing so yields the $E_8$ root lattice, whose
density is twice that of $D_8$.  (The $E_6$ and $E_7$ lattices are
certain cross sections of $E_8$.)

The $E_8$ and Leech lattices stand out among lattice packings, because
all the spheres fit beautifully into place in a remarkably dense and
symmetric way. There is no doubt that they are optimal packings in
general, not just among lattices.  Harmonic analysis ought to provide a
proof, but as we will see in Section~\ref{section:euclidean}, a full
proof has been elusive.

\emph{Periodic packings} form a broader class of packings than lattice
packings.  A lattice can be viewed as the vertices of a tiling of space
with parallelotopes (fundamental domains for the action by
translation), but there's no reason to center spheres only at the
vertices.  More generally, one can place them in the interior, or
elsewhere on the boundary, and then repeat them periodically; such a
packing is called a periodic packing.  Equivalently, the sphere centers
in a periodic packing form the union of finitely many translates of a
lattice.

The $E_8$ packing, as defined above, is clearly periodic (the union of
two translates of $D_8$).  It is not quite as obvious that it is
actually a lattice, but that is easy to check.  The Leech lattice in
$\R^{24}$ can be defined by a similar, but more elaborate, construction
involving filling in the holes in a lattice constructed using the
binary Golay code (see \cite{Le} and Section~4.4 in Chapter~4 of
\cite{CS}).

Philosophically, the construction of $E_8$ given above is somewhat odd,
because $E_8$ itself is extraordinarily symmetrical, but the
construction is not. Instead, it builds $E_8$ in two pieces.  This
situation is actually quite common when constructing a highly symmetric
object.  By neglecting part of the symmetry group, one can decompose
the object into simpler pieces, which can each be understood
separately. However, eventually one must exhibit the extra symmetry.
The symmetry group of $E_8$ is generated by the reflections in the
hyperplanes orthogonal to the minimal vectors of $E_8$, and one can
check that it acts transitively on those minimal vectors.

It is not known whether periodic packings achieve the maximal packing
density in every dimension.  However, they always come arbitrarily
close: given any dense packing, one can take a large, cubical piece of
it and repeat that piece periodically.  To avoid overlaps, it may be
necessary to remove some spheres near the boundary, but if the cube is
large enough, then the resulting decrease in density will be small.

By contrast, it is not even known whether there exist saturated lattice
packings in high dimensions.  If not, then lattices cannot achieve more
than half the maximal density, because one can double the density of a
non-saturated lattice by filling in a hole together with all its
translates by lattice vectors. It seems highly unlikely that there are
saturated lattices in high dimensions, because a lattice is specified
by a quadratic number of parameters, while there is an exponential
volume of space in which holes could appear, so there are not enough
degrees of freedom to control all the possible holes.  However, this
argument presumably cannot be made rigorous.

Despite all the reasons to think lattices are not the best sphere
packings in high dimensions, the best asymptotic lower bounds known for
sphere packing density use lattices. Ball's bound $2(n-1)2^{-n}$ in
$\R^n$ holds for lattice packings \cite{B}, and Vance's bound, which
improves it by an asymptotic factor of $3/e$ when $n$ is a multiple of
four, uses not just lattices, but lattices that are modules over a
maximal order in the quaternions \cite{V}.  Imposing algebraic
structure may rule out the densest possible packings, but it makes up
for that by offering powerful tools for analysis and proof.

\subsection{Packing problems in other spaces}

Packing problems are interesting in many metric spaces.  The simplest
situation is when the ambient space is compact, in which case the
packing will involve only finitely many balls.  The packing problem can
then be formulated in terms of two different optimization problems for
a finite subset of the metric space:
\begin{enumerate}
\item What is the largest possible minimal distance between $N$
    points?

\item What is the largest possible size of a subset whose minimal
    distance is at least $r$?
\end{enumerate}
The first fixes the number of balls and maximizes their size, while the
second fixes the radius $r/2$ of the balls and maximizes the number. In
Euclidean space, if we interpret the number of points as the number of
points per unit volume, then both problems are the same by scaling
invariance, but that does not hold in compact spaces. The two problems
are equivalent, however, in the sense that a complete answer to one
(for all values of $r$ or $N$) yields a complete answer to the other.

Packing problems arise naturally in many compact metric spaces,
including spheres, projective spaces, Grassmannians \cite{CHS,Ba}, and
the Hamming cube $\{0,1\}^n$ (under Hamming distance, so packings are
binary error-correcting codes). For a simplified example, suppose one
wishes to treat a spherical tumor by beaming radiation at it.  One
would like to use multiple beams approaching it from different angles,
so as to minimize radiation exposure outside of the tumor, and the
problem of maximizing the angle between the beams is a packing problem
in $\R\Ps^2$.

Packing problems are also important in non-compact spaces, but aside
from Euclidean space we will not deal with them in this article,
because defining density becomes much more subtle.  See, for example,
the foundational work by Bowen and Radin on defining packing density in
hyperbolic space \cite{BR}.

Packings on the surface of a sphere are known as spherical codes.
Specifically, an \emph{optimal spherical code} is an arrangement of
points on a sphere that maximizes the minimal distance among
configurations of its size. Spherical codes can be used as
error-correcting codes (for example, in the toy model of radio
transmission from Section~\ref{subsec:packinginformation}, they are
codes for a constant-power transmitter), and they also provide an
elegant way to help characterize the many interesting spherical
configurations that arise throughout mathematics.

One of the most attractive special cases of packing on a sphere is the
\emph{kissing problem}. How many non-overlapping unit balls can all be
tangent to a central unit ball?  The points of tangency on the central
ball form a spherical code with minimal angle at least $60^\circ$, and
any such code yields a kissing configuration.

In $\R^2$, the kissing number is clearly six, but the answer is already
not obvious in $\R^3$.  The twelve vertices of an icosahedron work, but
the tangent balls do not touch each other and can slide around.  It
turns out that there is no room for a thirteenth ball, but that was
first proved only in 1953 by Sch\"utte and van der Waerden \cite{SW}.

In $\R^4$, Musin \cite{M} showed that the kissing number is $24$, but
the answer is not known in $\R^5$ (it appears to be $40$). In fact, the
only higher dimensions for which the kissing problem has been solved
are $8$ and $24$, independently by Levenshtein \cite{Lev} and by
Odlyzko and Sloane \cite{OS}.  The kissing numbers are $240$ in $\R^8$
and $196560$ in $\R^{24}$. Furthermore, these kissing configurations
are unique up to isometries \cite{BS}.

The kissing number of $240$ is achieved by the $E_8$ root lattice
through its $240$ minimal vectors. Specifically, there are
$\binom{8}{2} \cdot 2^2 = 112$ permutations of $(\pm 1, \pm 1, 0,
\dots, 0)$ and $2^7 = 128$ vectors of the form $(\pm 1/2, \dots, \pm
1/2)$ with an even number of minus signs.  Thus, $E_8$ is not only the
densest lattice packing in $\R^8$, but it also has the highest possible
kissing number. Similarly, the Leech lattice in $\R^{24}$ achieves the
kissing number of $196560$.

In general, however, there is no reason to believe that the densest
packings will also have the highest kissing numbers.  The packing
density is a global property, while the kissing number is purely local
and might be maximized in a way that cannot be extended to a dense
packing. That appears to happen in many dimensions \cite{CS}. Instead
of being typical, compatibility between the optimal local and global
structures is a remarkable occurrence.

\section{The Thomson Problem and Universal Optimality}
\label{section:thomson}

\subsection{Physics on surfaces}

The Thomson problem \cite[p.~255]{Thomson} asks for the minimal-energy
configuration of $N$ classical electrons confined to the unit sphere
$S^2$. In other words, the particles interact via the Coulomb potential
$1/r$ at Euclidean distance $r$.  This model was originally intended to
describe atoms, before quantum mechanics or even the discovery of the
nucleus. Thomson hoped it would explain the periodic table. Of course,
subsequent discoveries have shown that it is a woefully inadequate
atomic model, but it remains of substantial scientific interest, and
its variants describe many real-world\break systems.

For example, imagine mixing together two immiscible liquids, such as
oil and water.  The oil will break up into tiny droplets, evenly
dispersed within the water, but they will rapidly coalesce and the oil
will separate from the water.  Cooks have long known that one can
prevent this separation by using emulsifiers.  One type of emulsion is
a Pickering emulsion, in which tiny particles collect on the boundaries
of oil droplets, which prevents coalescence (the particles bounce off
each other).

More generally, colloidal particles often adsorb to the interface
between two different liquids.  See, for example,
Figure~\ref{figure:chaikin}, which shows charged particles made of
polymethyl methacrylate (i.e., plexiglas) in a mixture of water and
cyclohexyl bromide. Notice that the particles on the surface of the
droplet have spread out into a fairly regular arrangement due to their
mutual repulsion, and they are repelling the remaining particles away
from the surface.

\begin{figure}
\begin{center}
\includegraphics[scale=0.35]{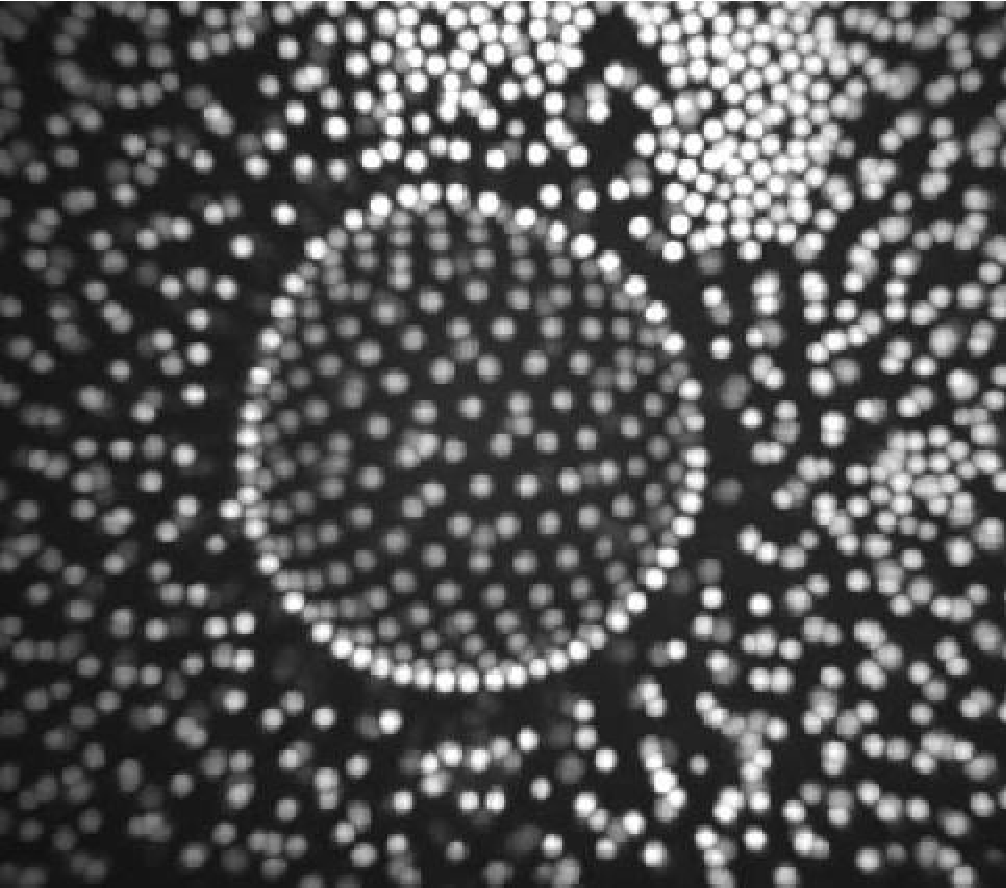}
\vspace*{-13pt}
\end{center}
\caption{Energy minimization on an actual (approximate) sphere:
tiny, electrically charged PMMA beads collecting on the interface
between water and cyclohexyl bromide. [Courtesy of W.~Irvine and
P.~M.~Chaikin, New York University]} \label{figure:chaikin}
\end{figure}

These particles are microscopic, yet large enough that they can
accurately be described using classical physics.  Thus, the generalized
Thomson problem is an appropriate model. See \cite{BG} for more details
on these sorts of materials.

Consider the case of particles on the unit sphere in $\R^n$. Given a
finite subset $\mathcal{C} \subset S^{n-1}$ and a \emph{potential
function} $f \colon (0,4] \to \R$, define the \emph{potential energy}
by
$$
E_f(\mathcal{C}) = \frac{1}{2}
\sum_{\shortstack[c]{$\scriptstyle x,y \in \mathcal{C}$\\
$\scriptstyle x \ne y$}} f\big(|x-y|^2\big).
$$
For each positive integer $N$ and each $f$, we seek an $N$-element
subset $\mathcal{C} \subset S^{n-1}$ that minimizes $E_f(\mathcal{C})$
compared to all other choices of $\mathcal{C}$ with $|\mathcal{C}|=N$.
The use of squared distance instead of distance is not standard in
physics, but it will prove mathematically convenient.  The function $f$
is defined only on $(0,4]$ because no squared distance larger than $4$
can occur on the unit sphere.

Typically $f$ will be decreasing (so the force is repulsive) and
convex.  In fact, the most natural potential functions to use are the
\emph{completely monotonic} functions, i.e., smooth functions
satisfying $(-1)^k f^{(k)} \ge 0$ for all integers $k \ge 0$.  For
example, inverse power laws $r \mapsto 1/r^s$ (with $s>0$) are
completely monotonic.

\subsection{Varying the potential function}

As we vary the potential function $f$ above, how do the optimal
configurations change?  From the physics perspective, this question
appears silly, because the potential is typically determined by
fundamental physics.  However, from a mathematical perspective it is a
critical question, because it places the individual optimization
problems into a richer context.

As we vary the potential function, the optimal configurations will vary
in some family.  This family may not be connected, because the optimum
may abruptly jump as the potential function passes some threshold, and
different components may have different dimensions \cite{CCEK}.
Nevertheless, we can use the local dimension of the family as a crude
measure of the complexity of an optimum: we compute the dimension of
the space of perturbed configurations that minimize energy for
perturbations of the potential function. Call this dimension the
\emph{parameter count} of the configuration.

\begin{figure}
\begin{center}
\includegraphics[scale=0.75]{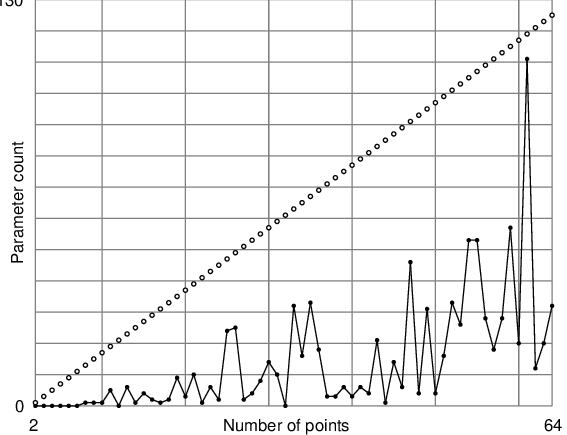}
\end{center}
\vspace*{-13pt}
\caption{Parameter counts for conjectural Coulomb-energy minimizers
on $S^2$.  For comparison, the white circles show the dimension of
the space of all configurations.}
\label{figure:parameter}
\end{figure}

Figure~\ref{figure:parameter} (taken from \cite{BBCGKS}) shows the
parameter counts for the configurations minimizing Coulomb energy on
$S^2$ with $2$ through $64$ points.  The figure is doubly conjectural:
in almost all of these cases, no proof is known that the supposed
optima are truly optimal or that the parameter counts are correct.
However, the experimental evidence leaves little doubt.

One can see from Figure~\ref{figure:parameter} that the parameter
counts vary wildly.  For example, for $43$ points there are $21$
parameters, while for $44$ points there is only $1$.  This suggests
that the $44$-point optimizer will be substantially simpler and more
understandable, and indeed it is (see \cite{BBCGKS}).

\subsection{Universal optimality}

When one varies the potential function, the simplest case is when the
optimal configuration never varies. Call a configuration
\emph{universally optimal} if it minimizes energy for all completely
monotonic potential functions.

A universal optimum is automatically an optimal spherical code: for the
potential function $f(r) = 1/r^s$ with $s$ large, the energy is
asymptotically determined by the minimal distance, and minimizing
energy requires maximizing the minimal distance.  However, optimal
spherical codes are rarely universally optimal.  For every number of
points in every dimension, there exists some optimal code, but
universal optima appear to be far less common.

In $S^1$, there is an $N$-point universal optimum for each $N$, namely
the vertices of a regular $N$-gon.  In $S^2$, the situation is more
complicated. Aside from degenerate cases with three or fewer points,
there are only three universal optima, namely the vertices of a regular
tetrahedron, octahedron, or icosahedron \cite{CK1}.  The cube and
dodecahedron are not even optimal, let alone universally optimal, since
one can lower energy by rotating a facet.

The first case for which there is no universal optimum is five points
in $S^2$. There are two natural configurations: a triangular bipyramid,
with an equilateral triangle on the equator together with the north and
south poles, and a square pyramid, with its top at the north pole and
its base slightly below the equator.  This second family depends on one
parameter, the height of the pyramid.  The triangular bipyramid is
known to minimize energy for several inverse power laws \cite{S}, but
it is not even a local minimum when they are sufficiently steep, in
which case square pyramids seem to become optimal.

\begin{conjecture}
For every completely monotonic potential function, either the
triangular bipyramid or a square pyramid minimizes energy among
five-point configurations in $S^2$.
\end{conjecture}

\begin{table}
\tabcolsep7.3pt
\vspace*{-10pt}
\caption{The known $N$-point universal optima in $S^{n-1}$.}
\label{table:univopt}
\begin{center}
\vspace*{-3pt}
\begin{tabular}{@{}c|c|c@{}}
\hline
$n$ & $N$ & Description\\
\hline
$n$ & $N \le n+1$ & regular simplex\\
$n$ & $2n$ & regular cross polytope\\
$2$ & $N$ & regular $N$-gon\\
$3$ & $12$ & regular icosahedron\\
$4$ & $120$ & regular $600$-cell\\
$5$ & $16$ & hemicube\\
$6$ & $27$ & Schl\"afli graph\\
$7$ & $56$ & $28$ equiangular lines\\
$8$ & $240$ & $E_8$ root system\\
$21$ & $112$ & isotropic subspaces\\
$21$ & $162$  & strongly regular graph\\
$22$ & $100$  & Higman-Sims graph\\
$22$ & $275$  & McLaughlin graph\\
$22$ & $891$  & isotropic subspaces\\
$23$ & $552$  & $276$ equiangular lines\\
$23$ & $4600$  & iterated kissing configuration\\
$24$ & $196560$& Leech lattice minimal vectors\\
$q(q^3+1)/(q+1)$ & $(q+1)(q^3+1)$ & isotropic subspaces ($q$ is a prime
power)\\
\hline
\end{tabular}
\end{center}
\vspace*{-10pt}
\end{table}

For $n \ge 4$, the universal optima in $S^{n-1}$ have not been
completely classified.  Table~\ref{table:univopt} shows a list of the
known cases (proved in \cite{CK1}). Each of them is a fascinating
mathematical object.  For example, the $27$ points in $S^5$ correspond
to the $27$ lines on a cubic surface.

The first five lines in the table list the regular polytopes with
simplicial facets.  The next four lines list the $E_8$ root system and
certain semiregular polytopes obtained as cross sections. The next
eight lines list the minimal vectors of the Leech lattice and certain
cross sections.  If this were the complete list, it would feel
reasonable, but the last line is perplexing.  It describes another
infinite sequence of universal optima, constructed from geometries over
$\F_q$ in \cite{CGS} and recognized as optimal codes in \cite{Lev2}.
How many more such cases remain to be constructed?

Another puzzling aspect of Table~\ref{table:univopt} is the gap between
$8$ and $21$ dimensions.  Are there really no universal optima in these
dimensions, aside from the simplices and cross polytopes?  Or do we
simply lack the imagination needed to discover them?  Extensive
computer searches \cite{BBCGKS} suggest that the table is closer to
complete than one might expect, but probably not complete.
Specifically, there are a $40$-point configuration in $S^9$ and a
$64$-point configuration in $S^{13}$ that appear to be universally
optimal, but these are the only conjectural cases that have been
located.

Almost all of the results tabulated in Table~\ref{table:univopt} can be
deduced from the following theorem.  It generalizes a theorem of
Levenshtein \cite{Lev2}, which says that these configurations are all
optimal codes.  The one known case not covered by the theorem is the
regular $600$-cell, which requires a different argument \cite{CK1}.

To state the theorem, we will need two definitions. A \emph{spherical
$k$-design} in $S^{n-1}$ is a finite subset $\mathcal{D}$ of the sphere
such that for every polynomial $p \colon \R^n \to \R$ of total degree
at most $k$, the average of $p$ over $\mathcal{D}$ equals its average
over the entire sphere.  Spherical $k$-designs can be thought of as
sets giving quadrature rules (i.e., numerical integration schemes) that
are exact for polynomials of degree up to $k$. An \emph{$m$-distance
set} is a set for which $m$ distances occur between distinct points.

\begin{theorem}[Cohn and Kumar \cite{CK1}] \label{theorem:univopt}
Every $m$-distance set that is a spherical $(2m-1)$-design is
universally optimal.
\end{theorem}

The proof of this theorem uses linear programming bounds, which are
developed in the next section.

\section{Proof Techniques: Linear Programming Bounds}
\label{section:compact}

\subsection{Constraints on the pair correlation function}
\label{subsec:constr}

In this section, we will discuss techniques for proving lower bounds on
potential energy.  In particular, we will develop linear programming
bounds and briefly explain how they are used to prove
Theorem~\ref{theorem:univopt}.

They are called ``linear programming bounds'' because linear
programming can be used to optimize them, but no knowledge of linear
programming is required to understand how the bounds work.  They were
originally developed by Delsarte for discrete problems in coding theory
\cite{D}, extended to continuous packing problems in \cite{DGS,KL}, and
adapted for potential energy minimization by Yudin and his
collaborators \cite{Y,KY1,KY2,A1,A2}. In this section, we will focus on
spherical configurations, although the techniques work in much greater
generality.

Given a finite subset $\mathcal{C}$ of $S^{n-1}$, define its
\emph{distance distribution} by
$$
A_t = \big|\{(x,y) \in \mathcal{C}^2 : \langle x,y \rangle = t\}\big|,
$$
where $\langle \cdot,\cdot \rangle$ denotes the inner product in
$\R^n$.  In physics terms, $A$ is the pair correlation function; it
measures how often each pairwise distance occurs (the inner product is
a natural way to gauge distance on the sphere). Linear programming
bounds are based on proving certain linear inequalities involving the
numbers $A_t$. These inequalities are crucial because the potential
energy can be expressed in terms of the distance distribution $A$ by
\begin{equation} \label{eq:Eflinear}
E_f(\mathcal{C}) = \frac{1}{2}
\sum_{\shortstack[c]{$\scriptstyle x,y \in \mathcal{C}$\\
$\scriptstyle x \ne y$}} f\big(|x-y|^2\big) = \sum_{-1 \le t < 1}
\frac{f(2-2t)}{2} A_t,
\end{equation}
since $|x-y|^2 = 2 - 2 \langle x,y \rangle$.  (Although
\eqref{eq:Eflinear} sums over uncountably many values of $t$, only
finitely many of the summands are nonzero.) Energy is a linear function
of $A$, and the linear programming bound is the minimum of this
function subject to the linear constraints on $A$, which makes it the
solution to a linear programming problem in infinitely many variables.

To begin, there are several obvious constraints on the distance
distribution.  Let $N = |\mathcal{C}|$.  Then $A_t \ge 0$ for all $t$,
$A_t=0$ for $|t|>1$, $A_1 = N$, and $ \sum_t A_t = N^2. $

The power of linear programming bounds comes from less obvious
constraints.  For example, $\sum_t A_t t \ge 0.$ To see why, notice
that
$$
\sum_t A_t t = \sum_{x,y \in \mathcal{C}} \langle x,y \rangle =
\left|\sum_{x \in \mathcal{C}} x\right|^2 \ge 0.
$$
More generally, there is an infinite sequence of polynomials
(independent of $\mathcal{C}$, but depending on the dimension $n$)
$P^n_0, P^n_1, P^n_2, \dots$, with $\deg P^n_k = k$, such that for each
$k$,
\begin{equation} \label{eq:nonnegsum}
\sum_t A_t P^n_k(t) \ge 0.
\end{equation}
(In fact, we can take $P_0^n(t)=1$, $P_1^n(t)=t$, and
$P_2^n(t)=t^2-1/n$.) This inequality is nontrivial, because these
polynomials are frequently negative. For example, $P^3_{12}$ looks like
this:
\begin{center}
\includegraphics[scale=0.75]{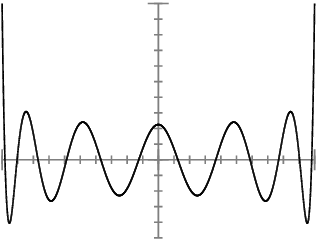}
\end{center}

The polynomials $P^n_k$ are called \emph{ultraspherical polynomials},
and they are characterized by orthogonality on the interval $[-1,1]$
with respect to the measure $(1-t^2)^{(n-3)/2} \, dt$.  In other words, for $i \ne
j$,
$$
\int_{-1}^1 P^n_i(t) P^n_j(t) \, (1-t^2)^{(n-3)/2} \, dt = 0.
$$
This relationship determines the polynomials up to scaling, as the
Gram-Schmidt orthogonalization of the monomials $1,t,t^2,\dots$ with
respect to this inner product. The sign of the scaling constant is
determined by $P^n_k(1)
> 0$, and the magnitude of the constant is irrelevant for
\eqref{eq:nonnegsum}.

In fact, these polynomials have a far stronger property than just
\eqref{eq:nonnegsum}: they are positive-definite kernels.  That is, for
any $N$ and any points $x_1,\dots,x_N \in S^{n-1}$, the $N \times N$
matrix $\big(P^n_k(\langle x_i, x_j \rangle)\big)_{1 \le i,j\le N}$ is
positive semidefinite.  This implies \eqref{eq:nonnegsum} because the
sum of the entries of a positive-semidefinite matrix is nonnegative.
Schoenberg \cite{Sch} proved that every continuous positive-definite
kernel on $S^{n-1}$ must be a nonnegative linear combination of
ultraspherical polynomials.

\subsection{Zonal spherical harmonics}

As an illustration of the role of representation theory, in this
section we will derive the ultraspherical polynomials as zonal
spherical harmonics and verify that they satisfy \eqref{eq:nonnegsum}.
The reader who is willing to take that on faith can skip the
derivation.

The orthogonal group $O(n)$ acts on $S^{n-1}$ by isometries, and hence
$L^2\big(S^{n-1}\big)$ is a unitary representation of $O(n)$.  To
begin, we will decompose this representation into irreducibles.  Let
$\Pol_k$ be the subspace of functions on $S^{n-1}$ defined by
polynomials on $\R^n$ of total degree at most $k$.  We have $\Pol_0
\subset \Pol_1 \subset \cdots$, and each $\Pol_k$ is a
finite-dimensional representation of $O(n)$, with $\bigcup_k \Pol_k$
dense in $L^2\big(S^{n-1}\big)$. To convert this filtration into a
direct sum decomposition, let $V_0=\Pol_0$ and define $V_k$ to be the
orthogonal complement of $V_0 \oplus V_1 \oplus \dots \oplus V_{k-1}$
within $\Pol_k$ (with respect to the usual inner product on
$L^2\big(S^{n-1}\big)$). Then $V_k$ is still preserved by $O(n)$, and
the entire space breaks up as
$$
L^2\big(S^{n-1}\big) = \mathop{\widehat{\bigoplus}}\limits_{k \ge 0} V_k.
$$
(The hat indicates the completion of the algebraic direct sum.) The
functions in $V_k$ are known as \emph{spherical harmonics} of degree
$k$, because $V_k$ is an eigenspace of the spherical Laplacian, but we
will not need that characterization of them.

For each $x \in S^{n-1}$, evaluating at $x$ defines a linear map $f
\mapsto f(x)$ on $V_k$.  Thus, there exists a unique vector $v_{k,x}
\in V_k$ such that for all $f \in V_k$,
$$
f(x) = \langle f, v_{k,x} \rangle,
$$
where $\langle \cdot,\cdot \rangle$ denotes the inner product on $V_k$
from $L^2\big(S^{n-1}\big)$. The map $x \mapsto v_{k,x}$ is called a
\emph{reproducing kernel}.

For each $T \in O(n)$ and $f \in V_k$,
$$
\langle f, v_{k,Tx}\rangle = f(Tx) = (T^{-1}f)(x) =
\langle T^{-1} f, v_{k,x} \rangle = \langle f, T v_{k,x} \rangle.
$$
Thus, $T v_{k,x} = v_{k,Tx}$, by the uniqueness of $v_{k,Tx}$.  It
follows that $v_{k,x}$ is invariant under the stabilizer of $x$ in
$O(n)$.  In other words, it is invariant under rotations about the axis
through $\pm x$, so it is effectively a function of only one variable,
the inner product with $x$.  Such a function is called a \emph{zonal
spherical harmonic}.

We can define $P^n_k$ by
$$
v_{k,x}(y) = P^n_k(\langle x,y \rangle).
$$
These polynomials certainly satisfy \eqref{eq:nonnegsum}, because
$$
\sum_{x,y \in \mathcal{C}} P^n_k(\langle x,y \rangle) =
\sum_{x,y \in \mathcal{C}} v_{k,x}(y) =
\sum_{x,y \in \mathcal{C}} \langle v_{k,x}, v_{k,y} \rangle =
\left|\sum_{x \in \mathcal{C}} v_{k,x}\right|^2 \ge 0,
$$
and in fact they are positive-definite kernels because
$\big(P^n_k(\langle x_i, x_j \rangle)\big)_{1 \le i,j\le N}$ is the
Gram matrix of the vectors $v_{k,x_i}$.

The functions $v_{0,x}, v_{1,x}, \dots$ are in orthogonal subspaces,
and hence the polynomials $P^n_0, P^n_1, \dots$ must be orthogonal with
respect to the measure on $[-1,1]$ obtained by projecting the surface
measure of $S^{n-1}$ onto the axis from $-x$ to $x$. The following
simple calculation shows that the measure is proportional to
$(1-t^2)^{(n-3)/2}\, dt$.  Consider the spherical shell defined by
$$
1 \le x_1^2 + \dots + x_n^2 \le 1+\varepsilon.
$$
If we set $x_1 = t$, then the remaining coordinates satisfy
$$
1-t^2 \le x_2^2 + \dots + x_n^2 \le 1-t^2 + \varepsilon,
$$
and the volume is proportional to $(1-t^2+\varepsilon)^{(n-1)/2} -
(1-t^2)^{(n-1)/2}$.  If we divide by $\varepsilon$ to normalize, then
as $\varepsilon \to 0$ we find that the density of the surface measure
with $x_1=t$ is proportional to $(1-t^2)^{(n-3)/2}$, as desired.

The degree of $P_k^n$ is at most $k$, and because $v_{k,x}$ is
orthogonal to $\Pol_{k-1}$, the degree can be less than $k$ only if
$P_k^n$ is identically zero.  That cannot be the case (for $n>1$),
since otherwise evaluating at $x$ would be identically zero.  If it
were, then it would follow from $T v_{k,x} = v_{k,Tx}$ that evaluating
at each point is identically zero, and thus that $V_k$ is trivial.
However, $\Pol_k \ne \Pol_{k-1}$, and hence $V_k$ is nontrivial.

Thus, the polynomials $P^n_k$ defined above have degree $k$, satisfy
\eqref{eq:nonnegsum}, and have the desired orthogonality relationship.

\subsection{Linear programming bounds}

Let $\mathcal{C} \subset S^{n-1}$ be a finite subset and let $A$ be its
distance distribution.  To make use of the linear constraints on $A$
discussed in Section~\ref{subsec:constr}, we will use the dual linear
program. In other words, we will take linear combinations of the
constraints so as to obtain a lower bound on energy.

We introduce new real variables $\alpha_k$ and $\beta_t$ specifying
which linear combination to take. Suppose we add $\alpha_0$ times
$\sum_t A_t = N^2$, $\alpha_k$ times
$$
\sum_{-1 \le t \le 1} A_t P^n_k(t) \ge 0
$$
(with $\alpha_k \ge 0$ for $k \ge 1$), and $\beta_t$ times the
constraint $A_t \ge 0$ (with $\beta_t \ge 0$ for $-1 \le t < 1$).  We
find that
$$
\sum_{-1 \le t \le 1} A_t \sum_k
\alpha_k P_k^n(t) + \sum_{-1 \le t < 1} A_t \beta_t \ge \alpha_0 N^2,
$$
using the normalization $P^n_0(t)=1$. Define $h(t) = \sum_k \alpha_k
P_k^n(t)$. Then
$$
\sum_{-1 \le t < 1} A_t \big(h(t) + \beta_t\big) \ge \alpha_0 N^2 - h(1) N,
$$
because $A_1 = N$. If we choose $\alpha_k$ and $\beta_t$ so that $h(t)
+ \beta_t = f(2-2t)/2$ for $-1 \le t < 1$, then the energy will be
bounded below by $\alpha_0 N^2 - h(1) N$, by \eqref{eq:Eflinear}.

The equation $h(t) + \beta_t = f(2-2t)/2$ just means that $h(t) \le
f(2-2t)/2$ (because we have assumed only that $\beta_t \ge 0$).  Thus,
we have proved the following bound:

\begin{theorem}[Yudin \cite{Y}] \label{theorem:yudin}
Suppose $h(t) = \sum_k \alpha_k P_k^n(t)$ satisfies $\alpha_k \ge 0$
for $k > 0$ and $h(t) \le f(2-2t)/2$ for $-1 \le t < 1$.  Then for
every finite subset $\mathcal{C} \subset S^{n-1}$,
$$
E_f(\mathcal{C}) \ge \alpha_0 |\mathcal{C}|^2 - h(1) |\mathcal{C}|.
$$
\end{theorem}

To prove Theorem~\ref{theorem:univopt}, one can optimize the choice of
the auxiliary function $h$ in Theorem~\ref{theorem:yudin}. Suppose
$\mathcal{C}$ is an $m$-distance set and a spherical $(2m-1)$-design,
and $f$ is completely monotonic.  In the proof of
Theorem~\ref{theorem:yudin}, equality holds if and only if $h(t) =
f(2-2t)/2$ for every inner product $t<1$ that occurs between points in
$\mathcal{C}$ and $\sum_{x,y \in \mathcal{C}} P^n_k(\langle x,y
\rangle) = 0$ whenever $\alpha_k>0$ and $k>0$.  The latter equation
automatically holds for $1 \le k \le 2m-1$ because $\mathcal{C}$ is a
$(2m-1)$-design.  Let $h$ be the unique polynomial of degree at most
$2m-1$ that agrees with $f(2-2t)/2$ to order $2$ at each of the $m$
inner products between distinct points in $\mathcal{C}$, so that $h$
satisfies the other condition for equality. The inequality $h(t) \le
f(2-2t)/2$ follows easily from a remainder theorem for Hermite
interpolation (using the complete monotonicity of $f$). The most
technical part of the proof is the verification that the coefficients
$\alpha_k$ of $h$ are nonnegative. For any single configuration, it can
be checked directly; for the general case, see \cite{CK1}.

\subsection{Semidefinite programming bounds}

Semidefinite programming bounds, introduced by Schrijver \cite{Schr}
and generalized by Bachoc and Vallentin \cite{BV}, extend the idea of
linear programming bounds by looking at triple (or even higher)
correlation functions, rather than just pair correlations.  Linear
constraints are naturally replaced with semidefinite constraints, and
the resulting bounds can be optimized by semidefinite programming.

This method is a far-reaching generalization of linear programming
bounds, and it has led to several sharp bounds that could not be
obtained previously \cite{BV2,CWoo}. However, the  improvement in the
bounds when going from pairs to triples is often small, while the
computational price is high. One of the most interesting conceptual
questions in this area is the trade-off between higher correlations and
improved bounds. When studying $N$-point configurations in $S^{n-1}$
using $k$-point correlation bounds, how large does $k$ need to be to
prove a sharp bound?  Clearly $k=N$ would suffice, and for the cases
covered by Theorem~\ref{theorem:univopt} it is enough to take $k=2$.
Aside from a handful of cases in which $k=3$ works, almost nothing is
known in between.  (Cases with $k \ge 4$ seem too difficult to handle
computationally.) This question is connected more generally to the
strength of LP and SDP hierarchies for relaxations of NP-hard
combinatorial optimization problems \cite{Lau}.

It is also related to a conjecture of Torquato and Stillinger
\cite{TS}, who propose that for packings that are disordered (in a
certain technical sense), in sufficiently high dimensions the two-point
constraints are not only necessary but also sufficient for the
existence of a packing with a given pair correlation function. They
show that this conjecture would lead to packings of density
$(1.715527\ldots+o(1))^{-n}$ in $\R^n$, by exhibiting the corresponding
pair correlation functions. The problem of finding a hypothetical pair
correlation function that maximizes the packing density, subject to the
two-point constraints, is dual to the problem of optimizing the linear
programming bounds.

\section{Euclidean Space} \label{section:euclidean}

\subsection{Linear programming bounds in Euclidean space}

Linear programming bounds can also be applied to packing and energy
minimization problems in Euclidean space, with Fourier analysis taking
the role played by the ultraspherical polynomials in the spherical
case. In this section, we will focus primarily on packing, before
commenting on energy minimization at the end.  The theory is formally
analogous to that in compact spaces, but the resulting optimization
problems are quite a bit deeper and more subtle, and the most exciting
applications of the theory remain conjectures.

We will normalize the Fourier transform of an $L^1$ function $f \colon
\R^n \to \R$ by
$$
\widehat{f}(t) = \int_{\R^n} f(x) e^{2\pi i \langle t,x \rangle } \, dx.
$$
(In this section, $f$ will not denote a potential function.) The
fundamental technical tool is the Poisson summation formula for a
lattice $\Lambda$, which holds for all Schwartz functions (i.e., smooth
functions all of whose derivatives are rapidly decreasing):
$$
\sum_{x \in \Lambda} f(x) =
\frac{1}{\vol(\R^n/\Lambda)} \sum_{t \in \Lambda^*} \widehat{f}(t).
$$
Here, $\vol(\R^n/\Lambda)$ is the volume of a fundamental
parallelotope, and $\Lambda^*$ is the dual lattice defined by
$$
\Lambda^* = \{ t \in \R^n : \langle t,x \rangle \in \Z
\textup{ for all $x \in \Lambda$}\}.
$$
Given any basis of $\Lambda$, the dual basis with respect to $\langle
\cdot, \cdot \rangle$ is a basis of $\Lambda^*$.

\begin{theorem}[Cohn and Elkies \cite{CE}] \label{theorem:euclidean}
Let $f \colon \R^n \to \R$ be a Schwartz function such that
$\widehat{f}(0) \ne 0$. If $f(x) \le 0$ for $|x| \ge 1$ and
$\widehat{f}(t) \ge 0$ for all $t$, then the sphere packing density in
$\R^n$ is at most
$$
\frac{\pi^{n/2}}{2^n(n/2)!} \cdot \frac{f(0)}{\widehat{f}(0)}.
$$
\end{theorem}

Of course, $(n/2)!$ means $\Gamma(n/2+1)$ when $n$ is odd.  The
restriction to Schwartz functions can be replaced with milder
assumptions \cite{CE,CK1}.

The hypotheses and conclusion of Theorem~\ref{theorem:euclidean} are
invariant under rotation about the origin, so without loss of
generality we can symmetrize $f$ and assume it is a radial function.
Thus, optimizing the bound in Theorem~\ref{theorem:euclidean} amounts
to optimizing the choice of a function of one (radial) variable.

It is not hard to prove Theorem~\ref{theorem:euclidean} for the special
case of lattice packings.  Suppose $\Lambda$ is a lattice, and rescale
so we can assume the minimal vector length is $1$ (i.e., the packing
uses balls of radius $1/2$).  The density is the volume of a sphere of
radius $1/2$, which is $\pi^{n/2}/(2^n(n/2)!)$, times the number of
spheres occurring per unit volume in space.  The latter factor is
$1/\vol(\R^n/\Lambda)$, because there is one sphere for each
fundamental cell of the lattice, and hence the density equals
$$
\frac{\pi^{n/2}}{2^n(n/2)!} \cdot \frac{1}{\vol(\R^n/\Lambda)}.
$$

Now we apply Poisson summation to see that
$$
\sum_{x \in \Lambda} f(x) =
\frac{1}{\vol(\R^n/\Lambda)} \sum_{t \in \Lambda^*} \widehat{f}(t).
$$
The left side is bounded above by $f(0)$, because all the other terms
come from $|x| \ge 1$ and are thus nonpositive by assumption.  The
right side is bounded below by $\widehat{f}(0)/\vol(\R^n/\Lambda)$,
because all the other terms are nonnegative.  Thus,
$$
f(0) \ge
\frac{\widehat{f}(0)}{\vol(\R^n/\Lambda)},
$$
which is equivalent to the
density bound in Theorem~\ref{theorem:euclidean}.

The proof in the general case is completely analogous.  It suffices to
prove the bound for all periodic packings (because they come
arbitrarily close to the maximal density), and one can apply a version
of Poisson summation for summing over translates of a lattice. See
\cite{CE} for the details, as well as for an explanation of the analogy
between these linear programming bounds and those for compact spaces.

\subsection{Apparent optimality of $E_8$ and the Leech lattice}

Theorem~\ref{theorem:euclidean} does not explain how to choose the
function $f$, and for $n>1$ the optimal choice of $f$ is unknown.
However, one can use numerical methods to optimize the density bound,
for example by choosing $f(x)$ to be $e^{-\pi|x|^2}$ times a polynomial
in $|x|^2$ (so that the Fourier transform can be easily computed) and
then optimizing the choice of the polynomial.  For $4 \le n \le 36$,
the results were collected in Table~3 of \cite{CE}, and in each case
the bound is the best one known, but they are typically nowhere near
sharp. For example, when $n=36$, the upper bound is roughly $58.2$
times the best packing density known.  That was an improvement on the
previous bound, which was off by a factor of $89.7$, but the gap
remains enormous.

However, for $n=2$, $8$, or $24$, Theorem~\ref{theorem:euclidean}
appears to be sharp:

\begin{conjecture}[Cohn and Elkies \cite{CE}] \label{conjecture:euclidean}
For $n=2$, $8$, or $24$, there exists a function $f$ that proves a
sharp bound in Theorem~\ref{theorem:euclidean} (for the hexagonal,
$E_8$, or Leech lattice, respectively).
\end{conjecture}

The strongest numerical evidence comes from \cite{CK2}: for $n=24$ the
bound is sharp to within a factor of $1 + 1.65 \cdot 10^{-30}$. Similar
accuracy can be obtained for $n=8$ or $n=2$, although only $10^{-15}$
was reported in \cite{CK2}. Of course, for $n=2$ the sphere packing
problem has already been solved, but
Conjecture~\ref{conjecture:euclidean} is open.

This apparent sharpness is analogous to the sharpness of the linear
programming bounds for the kissing number in $\R^2$, $\R^8$, and
$\R^{24}$. In that problem, it would have sufficed to prove any bound
less than the answer plus one, because the kissing number must be an
integer, but the bounds in fact turn out to be exact integers.  In the
case of the sphere packing problem, the analogous exactness is needed
(because packing density is not quantized), and fortunately it appears
to be true.

Examining the proof of Theorem~\ref{theorem:euclidean} gives simple
conditions for when the bound can be sharp for a lattice $\Lambda$,
analogous to the conditions for Theorem~\ref{theorem:yudin}: $f$ must
vanish at each nonzero point in $\Lambda$ and $\widehat{f}$ must vanish
at each nonzero point in $\Lambda^*$.  In fact, the same must be true
for all rotations of $\Lambda$, so $f$ and $\widehat{f}$ must vanish at
these radii (even if they are not radial functions). Unfortunately, it
seems difficult to control the behavior of $f$ and $\widehat{f}$
simultaneously.

For the special case of lattices, however, it is possible to complete a
proof.

\begin{theorem}[Cohn and Kumar \cite{CK2}] \label{theorem:leech}
The Leech lattice is the unique densest lattice in $\R^{24}$, up to
scaling and isometries.
\end{theorem}

The proof uses Theorem~\ref{theorem:euclidean} to show that no sphere
packing in $\R^{24}$ can be more than slightly denser than the Leech
lattice, and that every lattice as dense as the Leech lattice must be
very close to it.  However, the Leech lattice is a locally optimal
packing among lattices, and the bounds can be made close enough to
complete the proof. This approach also yields a new proof of optimality
and uniqueness for $E_8$ (previously shown in \cite{Bl} and \cite{Ve}).

One noteworthy hint regarding the optimal functions $f$ in $\R^8$ and
$\R^{24}$ is an observation of Cohn and Miller \cite{CM} about the
Taylor series coefficients of $f$.  It is more convenient to use the
rescaled function $g(x) = f(x/r)$, where $r = \sqrt{2}$ when $n=8$ and
$r=2$ when $n=24$. Then $g(0) = \widehat{g}(0)$, and without loss of
generality let this value be $1$. Assuming $g$ is radial, we can view
$g$ and $\widehat{g}$ as functions of one variable and ask for their
Taylor series coefficients.  Only even exponents occur by radial
symmetry, so the first nontrivial terms are quadratic. Cohn and Miller
noticed that the quadratic coefficients appear to be rational numbers,
as shown in Table~\ref{table:coeffs}. The quartic terms seem more
subtle, and it is not clear whether they are rational as well. If they
are, then their denominators are probably much larger.

\begin{table}
\tabcolsep8.5pt
\vspace*{-10pt}
\caption{Approximate Taylor series coefficients of $g$ and
$\widehat{g}$ about $0$.} \label{table:coeffs}
\begin{center}
\begin{tabular}{@{}ccccc@{}}
\hline
$n$ & function & order & coefficient & conjecture\\
\hline $8$ & $g$ & $2$ & $-2.7000000000000000000000000\dots$ & $-27/10$\\
$8$ & $\widehat{g}$ & $2$ & $-1.5000000000000000000000000\dots$ & $-3/2$\\
$24$ & $g$ & $2$ & $\phantom{-}2.6276556776556776556776556\dots$ & $14347/5460$\\
$24$ & $\widehat{g}$ & $2$ & $\phantom{-}1.3141025641025641025641025\dots$ & $205/156$\\
$8$ & $g$ & $4$ & $\phantom{-}4.2167501240968298210999141\dots$ & ?\\
$8$ & $\widehat{g}$ & $4$ & $-1.2397969070295980026220772\dots$ & ?\\
$24$ & $g$ & $4$ & $\phantom{-}3.8619903167183007758184168\dots$ & ?\\
$24$ & $\widehat{g}$ & $4$ & $\phantom{-}0.7376727789015322303799539\dots$ & ?\\
\hline
\end{tabular}
\end{center}
\vspace*{-10pt}
\end{table}

More generally, one can study not just the sphere packing problem, but
also potential energy minimization in Euclidean space.  The total
energy of a periodic configuration will be infinite, because each
distance occurs infinitely many times, but one can instead try to
minimize the average energy per particle.  Some of the densest packings
minimize more general forms of energy, but others do not, and
simulations lead to many intriguing structures \cite{CKS}.

Cohn and Kumar \cite{CK1} proved linear programming bounds for energy
and made a conjecture analogous to
Conjecture~\ref{conjecture:euclidean}:

\begin{conjecture}[Cohn and Kumar \cite{CK1}]
\label{conjecture:eucunivopt} For $n=2$, $8$, or $24$, the linear
programming bounds for potential energy minimization in $\R^n$ are
sharp for every completely monotonic potential function (for the
hexagonal, $E_8$, or Leech lattice, respectively).
\end{conjecture}

This universal optimality would be a dramatic strengthening of mere
optimality as packings.  It is not even known in the two-dimensional
case.

\section{Future Prospects} \label{section:future}

The most pressing question raised by this work is how to prove that the
hexagonal lattice, $E_8$, and the Leech lattice are universally optimal
in Euclidean space.  Linear programming bounds reduce this problem to
finding certain auxiliary functions of one variable, and the optimal
functions can even be computed to high precision, but so far there is
no proof that they truly exist.

More generally, can we classify the universal optima in a given space?
No proof is known even that the list of examples in $S^3$ is complete,
although it very likely is.  Each of the known universal optima is such
a remarkable mathematical object that a classification would be highly
desirable: if there are any others out there, we ought to find them.

One noteworthy case is equiangular line configurations in complex
space.  Do there exist $n^2$ unit vectors $x_1,\dots,x_{n^2} \in \C^n$
such that for $i \ne j$, $|\langle x_i, x_j \rangle|^2$ is independent
of $i$ and $j$ (in which case one can show it must be $1/(n+1)$)?  In
other words, the complex lines through these vectors are equidistant
under the Fubini-Study metric in $\C\Ps^{n-1}$.  Zauner \cite{Z}
conjectured that the answer is yes for all $n$, and substantial
numerical evidence supports that conjecture \cite{RBSC}, but only
finitely many cases have been proved. A collection of $n^2$ vectors
with this property gives an $n^2$-point universal optimum in
$\C\Ps^{n-1}$, by Theorem~8.2 in \cite{CK1}.  This case is particularly
unusual, because normally the difficulty is in proving optimality for a
configuration that has already been constructed, rather than
constructing one that has already been proved optimal (should it
exist).

These equiangular line configurations are in fact closely analogous to
Hadamard matrices.  They can be characterized as exactly the simplices
in $\C\Ps^{n-1}$ that are projective $2$-designs (where a simplex is
simply a set of points for which all pairwise distances are equal).
Similarly, Hadamard designs, which are an equivalent variant of
Hadamard matrices \cite{AK}, are symmetric block $2$-designs that are
simplices under the Hamming distance between blocks.  The existence of
Hadamard matrices of all orders divisible by four is a famous unsolved
problem in combinatorics, and perhaps the problem of $n^2$ equiangular
lines in $\C^n$ will be equally difficult.

These two problems are finely balanced between order and disorder.  Any
Hadamard matrix or equiangular line configuration must have
considerable structure, but in practice they frequently seem to have
just enough structure to be tantalizing, without enough to guarantee a
clear construction.  This contrasts with many of the most symmetrical
mathematical objects, which are characterized by their symmetry groups:
once you know the full group and the stabilizer of a point, it is often
not hard to deduce the structure of the complete object.  That seems
not to be possible in either of these two problems, and it stands as a
challenge to find techniques that can circumvent this difficulty.

In conclusion, packing and energy minimization problems exhibit greatly
varying degrees of symmetry and order in their solutions.  In certain
cases, the solutions are extraordinary mathematical objects such as
$E_8$ or the Leech lattice.  Sometimes this can be proved, and
sometimes it comes down to simply stated yet elusive conjectures.  In
other cases, the solutions may contain defects or involve unexpectedly
complicated structures.  Numerical experiments suggest that this is the
default behavior, but it is difficult to predict exactly when or how it
will occur. Finally, in rare cases there appears to be order of an
unusually subtle type, as in the complex equiangular line problem, and
this type of order remains a mystery.

\section*{Acknowledgments}

I am grateful to James Bernhard, Tom Brennan, Tzu-Yi Chen, Donald Cohn,
Noam Elkies, Abhinav Kumar, Achill Sch\"urmann, Sal Torquato, Frank
Vallentin, Stephanie Vance, Jeechul Woo, and especially Nadia Heninger
for their valuable feedback on this paper.

\end{document}